\documentclass[11pt]{amsart}

\usepackage{amscd,amsmath,amssymb,fancyhdr,color}
\usepackage[utf8x]{inputenc}
\usepackage{amsfonts}
\usepackage{amsthm}
\usepackage{accents}
\usepackage{graphicx}
\usepackage{float}
\usepackage{subcaption}
\usepackage{verbatim}
\usepackage{indentfirst}
\usepackage{dsfont}
\usepackage{tikz}
\usepackage{enumerate}
\usepackage[all,cmtip]{xy}
\usepackage{faktor}
\usepackage{lmodern}
\usepackage{mathtools}

\usepackage[backref=page]{hyperref}
\renewcommand*{\backref}[1]{}
\renewcommand*{\backrefalt}[4]{%
    \ifcase #1 (Not cited.)%
    \or        (Cited on page~#2.)%
    \else      (Cited on pages~#2.)%
    \fi}

\hypersetup{
     colorlinks   = true,
     citecolor    = magenta
}

\newcommand{\reg}{\operatorname{reg}}
\newcommand{\sing}{\operatorname{sing}}

\newcommand{\K}{K\"ahler}

\numberwithin{equation}{section}

\def\eqref#1{(\ref{#1})}

\def\1{\sqrt{-1}\:}

\newcommand{\cntrct}                
{\hspace{2pt}\raisebox{1pt}{\text{$\lrcorner$}}\hspace{2pt}}

\newcommand{\Aut}{\operatorname{Aut}}

\newcommand{\diag}{\operatorname{diag}}

\renewcommand{\Re}{\operatorname{Re}}

\newcounter{Mycounter}[section]
\newcounter{lemma}[section]
\newcounter{claim}[section]

\newcounter{sublemma}[section]

\newcounter{corollary}[section]

\newcounter{theorem}[section]

\newcounter{conjecture}[section]

\newcounter{proposition}[section]

\newcounter{definition}[section]

\newcounter{example}[section]

\newcounter{remark}[section]

\newcounter{problem}[section]

\newcounter{question}[section]

\makeatletter

\@addtoreset{equation}{section}

\@addtoreset{footnote}{section}

\makeatother

\usetikzlibrary{arrows,chains,matrix,positioning,scopes}

\makeatletter
\tikzset{join/.code=\tikzset{after node path={%
			\ifx\tikzchainprevious\pgfutil@empty\else(\tikzchainprevious)%
			edge[every join]#1(\tikzchaincurrent)\fi}}}
\makeatother

\tikzset{>=stealth',every on chain/.append style={join},
	every join/.style={->}}

\begin{document}

\title{LCK metrics on complex spaces with quotient singularities}

\author{George-Ionu\c t Ioni\c t\u a}

\address{%
\textsc{Department of Mathematics and Computer Science\\ 
"Politehnica" University Bucharest}\\ 
313 Splaiul Independen\c tei, Bucharest 060042, Romania}

\email{georgeionutionita@gmail.com}

\author{Ovidiu Preda}

\address{%
\textsc{Institute of Mathematics of the Romanian Academy}\\
P.O. Box 1-764, Bucharest 014700, Romania}

\email{ovidiu.preda@imar.ro}

\thanks{Ovidiu Preda was
supported by a grant of Ministry of Research and Innovation, CNCS - UEFISCDI, project number
PN-III-P1-1.1-PD-2016-0182, within PNCDI III}

\subjclass[2010]{32C15; 53C55}

\begin{abstract}
In this article we introduce a generalization of locally conformally \K\ metrics from complex manifolds to complex analytic spaces with singularities and study which properties of locally conformally \K\ manifolds still hold in this new setting. We prove that if a complex analytic space has only quotient singularities, then it admits a locally conformally \K\   metric if and only if its universal cover admits a \K\ metric such that the deck automorphisms act by homotheties of the \K\ metric. We also prove that the blow-up at a point of an LCK complex space is also LCK.
\end{abstract}

\maketitle
\section{Introduction}

A \K\ manifold is a complex manifold admitting a (1,1)-form $\omega$ which is positive definite and $\mathrm{d}$-closed. This form is called \K\ form, or by an abuse of language which is unlikely to cause any confusion, \K\ metric, since it corresponds to a hermitian metric.
By Dolbeault's lemma, a \K\ form can be written locally $\omega=i\partial\overline{\partial}\varphi$, where $\varphi$ is a strictly plurisubharmonic function, called \K\ potential. Hence, $\omega$ is completely determined by a family of \K\ potentials $(\varphi_\alpha)_{\alpha\in A}$, which verify the compatibility condition $i\partial\overline{\partial}\varphi_\alpha=i\partial\overline{\partial}\varphi_\beta$ on the open subset where both are defined. Grauert \cite{GRA} and Moishezon \cite{MOI} used this equivalent definition to extend the notion of \K\ metrics to complex spaces with singularities. 

A locally conformally \K\ manifold $M$ is a complex manifold admitting a (1,1)-form $\omega$ such that every point $x\in M$ has a neighborhood $U$ and there is a smooth function $f:U\rightarrow \mathbb{R}$ such that $e^{-f}\omega$ is \K .
By using local \K\ potentials and compatibility conditions, the definition of locally conformally \K\ metrics can be extended to complex spaces with singularities in the same manner as Grauert and Moishezon did for the \K\ case.

A well known characterization of locally conformally \K\ manifolds (LCK for short) is the following result: a complex manifold $M$ admits an LCK metric if and only if its universal cover $\widetilde{M}$ admits a \K\ metric such that the deck automorphisms act on $\widetilde{M}$ by \K\ homotheties. 

Therefore, it is natural to ask wether or not the equivalent definition for LCK remains valid for complex analytic spaces. The main result of this article solves this affirmatively in the case of complex analytic spaces with quotient singularities:

\begin{theorem}\label{TH1}
Let $X$ be a complex analytic space which has only quotient singularities. 

Then, $X$ admits an LCK metric if and only if its universal cover $\widetilde{X}$ admits a \K\ metric such that deck automorphisms act on $\widetilde{X}$ by homotheties of the \K\ metric.
\end{theorem}

\hspace{0.4cm}

We also generalize to complex spaces a classical theorem regarding LCK manifolds. We prove the following result:

\begin{theorem}\label{TH2}
The blow-up at a point of an LCK complex space is also LCK.
\end{theorem}

\hspace{0.4cm}

In section 4 of this paper, we give some examples of LCK singular complex spaces which do not admit \K\ metrics. They are obtained as a quotient of an LCK manifold by a finite group of automorphisms with fixed points. 
In the last section, we make some remarks and propose an open problem, for further study of LCK complex spaces.

\section{Preliminaries}

In this section, we collect the notions, definitions, and results that we need for the main theorem. 

\begin{definition}
Let $X$ be a complex space. A \textit{K\" ahler metric} on $X$ is a collection $(U_\alpha, \varphi_\alpha)_{\alpha\in A}$, where $(U_\alpha)_{\alpha\in A}$ is an open covering of $X$, $\varphi_\alpha$ is a strongly plurisubharmonic function on $U_\alpha$, such that on each nonempty intersection $U_\alpha\cap U_\beta$ we have the following compatibility condition: 
$$\varphi_\alpha-\varphi_\beta=\Re g_{\alpha\beta},$$ 
where $g_{\alpha\beta}$ is a holomorphic function on $U_\alpha\cap U_\beta$. 
\end{definition}

If $X$ is a complex manifold, such a collection $(U_\alpha, \varphi_\alpha)_{\alpha\in A}$ defines indeed a \K\ form on $X$, given locally, on each set $U_\alpha$, by $i\partial\overline{\partial}\varphi_\alpha$.

If the collection $(U_\alpha, \varphi_\alpha)_{\alpha\in A}$ verifies the open covering and the compatibility conditions from the \K\ metric definition above, but each function $\varphi_\alpha$ is only assumed to be plurisubharmonic, and strictly plurisubharmonic on the complement of an analytic subset of positive codimension in $U_\alpha$, then $(U_\alpha, \varphi_\alpha)_{\alpha\in A}$ is called a \textit{weakly \K\ metric} on $X$.

\vspace{0,4cm}
For a short presentation of locally conformally \K\ (LCK) manifolds, one may read the survey \cite{OV} by Ornea and Verbitsky. There are more equivalent definitions of LCK manifolds:
\begin{definition}
A complex manifold $M$ is called LCK if it verifies one of the following equivalent conditions:
\begin{enumerate}
\item There exists a $(1,1)$-form $\omega$ on $M$ such that for every $x\in M$, there exists an open neighborhood $U$ of $x$ and a smooth function $f:U\rightarrow \mathbb{R}$ such that $e^{-f}\omega$ is a \K\ form on $U$;
\item $M$ has a Hermitian metric $\omega$ such that $\mathrm{d}\omega=\theta\wedge\omega$, where $\theta$ is a closed 1-form on $M$, called the \textit{Lee form};
\item The universal cover $\widetilde{M}$ of $M$ has a \K\ metric such that the deck transform group acts on $\widetilde{M}$ by \K\ homotheties;
\item $M$ admits an oriented, flat, real line bundle $(L,\nabla)$ and an $L$-valued (1,1)-form $\omega$ which is \K\ with respect to $d_\nabla$.
\end{enumerate}
\end{definition}

Using the first of these equivalent definitions, we can generalize LCK metrics to singular complex spaces, in the following way:
\begin{definition}
Let $X$ be a complex space. An \textit{LCK metric} on $X$ is a collection $(U_\alpha, \varphi_\alpha, h_\alpha)_{\alpha\in A}$, where $(U_\alpha)_{\alpha\in A}$ is an open covering of $X$, $\varphi_\alpha$ is a strongly plurisubharmonic function on $U_\alpha$, and $h_\alpha:U_\alpha\rightarrow \mathbb{R}$ is a smooth function, such that we have the following compatibility condition: 
$$e^{h_\alpha}i\partial\overline{\partial}\varphi_\alpha=e^{h_\beta}i\partial\overline{\partial}\varphi_\beta$$ 
on each nonempty intersection $U_\alpha\cap U_\beta\cap X_{\reg}$.

As in the \K\ case, such a collection $(U_\alpha, \varphi_\alpha, h_\alpha)_{\alpha\in A}$ will be called \textit{weakly LCK} if we require every function $\varphi_\alpha$ to be only plurisubharmonic, and strictly plurisubharmonic on the complement of an analytic subset of positive codimension in $U_\alpha$.
\end{definition}

\vspace{0.4cm}
Two \K\ (or LCK) metrics are considered to be equal if they determine the same (1,1)-form on the regular locus of the complex space.

If $(X,\omega)$ is a \K\ space, $\omega=(U_\alpha,\varphi_\alpha)_{\alpha\in A}$, and $h:X\rightarrow \mathbb{R}$ is a smooth function, we denote by $e^h\omega$ the metric $(U_\alpha, \varphi_\alpha, h_{|U_\alpha})_{\alpha\in A}$. We say that an automorphism $\gamma\in \Aut(X)$ acts by homotheties of the \K\ metric if $\gamma^\star\omega=e^C\omega$, where $C\in \mathbb{R}$.

\begin{definition}
If $X$ is a complex space, then a point $x\in X_{\sing}$ is called \textit{quotient singularity} if there exists a finite subgroup $G$ of automorphisms of $\mathbb{C}^n$ such that the germs $(X,x)$ and $(\mathbb{C}^n/G,0)$ are biholomorphic.
\end{definition}

\vspace{0.4cm}
The folowing result about the properties of quotient singularities is a particular case of H. Cartan's \cite[Th\' eor\` eme 1]{CAR}:

\begin{theorem}\label{TH_CAR}
Let $M$ be a complex manifold and $G$ a finite group of automorphisms of $M$.

Then, the quotient space $X:=M/G$ with the sheaf induced by the canonical projection $p:M\rightarrow M/G$ is a normal space.
\end{theorem}

\vspace{0.4cm}
The next theorem, by Bierstone and Milman \cite[Theorem 13.4]{BI-MI}, is the fundamental result on global desingularization of complex spaces.
\begin{theorem}
Any complex space $X$ admits a desingularization $\pi:\widetilde{X}\rightarrow X$ such that $\pi$ is the composition of a locally finite sequence of blow-ups with smooth centers and $\pi^{-1}(X_{\sing})$ is a divisor with normal crossings in $\widetilde{X}$.
\end{theorem}

In this theorem locally finite sequence of blow-ups means that on every compact subset, all but finitely many blow-ups are trivial.

\vspace{0.4cm}
The following theorem of Koll\' ar, which combines \cite[Lemma 7.2]{KOL} and \cite[Theorem 7.5]{KOL}, gives a sufficient condition under which the fundamental group of a normal space and the fundamental group of a desingularization of it, are isomorphic.

\begin{theorem}\label{TH_KOL}
Let $X$ be a normal space which has only quotient singularities and $f:Y\rightarrow X$ a resolution of singularities. 

Then, the induced homomorphism $f_\star:\pi_1(Y)\rightarrow \pi_1(X)$ is an isomorphism.
\end{theorem}

Of course, taking into account \ref{TH_CAR}, the assumption of normality in Koll\' ar's theorem is superfluous, but we kept the original statement. We mention that a different proof for \ref{TH_KOL} was given by Verbitsky \cite[Theorem 4.1]{VER}.


\section{The main results}

\subsection{A characterization theorem for LCK complex spaces}

\begin{proof}[Proof of \ref{TH1}]
Firstly, we prove the direct implication, so we know by hypothesis that $X$ admits an LCK metric.
We denote by  $p:\widetilde{X}\rightarrow X$ the universal cover of $X$ and we consider $\pi:Y\rightarrow X$ a resolution of singularities for $X$. These two maps  induce a resolution of singularities $\widetilde{\pi}:\widetilde{Y}\rightarrow \widetilde{X}$ and a cover $\widetilde{p}:\widetilde{Y}\rightarrow Y$ such that the following diagram commutes:
$$\xymatrix{
\widetilde{Y} \ar[d]_-{\widetilde{\pi}} \ar[r]^{\widetilde{p}} &Y\ar[d]^-{\pi}\\
\widetilde{X} \ar[r]_p          &X}$$

Now, denote by $(U_{\alpha},\varphi_{\alpha},h_{\alpha})_{\alpha \in A}$ the LCK metric on $X$. Then, with the notations: $V_\alpha:=\pi^{-1}(U_{\alpha})$, $\psi_\alpha:=\varphi_{\alpha}\circ\pi$, and $g_\alpha:=h_{\alpha}\circ\pi$ for every ${\alpha \in A}$, we have that $(V_{\alpha},\psi_{\alpha},g_{\alpha})_{\alpha \in A}$ is a weakly LCK metric on $Y$. Since $p:\widetilde{Y}\rightarrow Y$ is a covering of $Y$, if we define $\widetilde{V}_\alpha:=\widetilde{p}^{-1}(V_\alpha)$, $\widetilde{\psi}_\alpha:=\psi_\alpha\circ\widetilde{p}$, and $\widetilde{g}_\alpha:=g_\alpha\circ\widetilde{p}$, we obtain that 
$(\widetilde{V}_\alpha, \widetilde{\psi}_\alpha, \widetilde{g}_\alpha)_{\alpha\in A}$ is a weakly LCK metric on $\widetilde{Y}$. 
Denote by $\widetilde{\theta}$ its induced Lee form. 
Since $\widetilde{X}$ is the universal cover of a complex space with only quotient singularities, it also has only quotient singularities, and by \ref{TH_CAR} it is also normal. Hence, by \ref{TH_KOL}, $\widetilde{Y}$ is simply connected, which further implies that $\widetilde{\theta}$ is exact: there exists $F\in\mathcal{C}^\infty(\widetilde{Y})$ such that $\widetilde{\theta}= \mathrm{d}F$.

Next, we may assume that the sets $(U_\alpha)_{\alpha\in A}$ of the LCK structure on $X$ are connected and sufficiently small such that for each $\alpha\in A$, $p^{-1}(U_\alpha)$ is a disjoint union of open sets in $\widetilde{X}$, each of them biholomorphic to $U_\alpha$. Then, for each $\alpha\in A$, $\widetilde{p}^{-1}(V_\alpha)=\cup_{i\in I_\alpha} \widetilde{V}_{\alpha, i}$ is a union of pairwise disjoint open connected sets, each of them biholomorphic to $V_\alpha$, and we denote $\widetilde{U}_{\alpha,i}:=\widetilde{\pi}(\widetilde{V}_{\alpha,i})$; for every $\alpha\in A$ and $i\in I_\alpha$, we have that $p(\widetilde{U}_{\alpha,i})=U_\alpha$. Also, for every $\alpha\in A$ and every $i\in I_\alpha$, 
$\widetilde{p}: \widetilde{V}_{\alpha, i} \rightarrow V_\alpha$ is a biholomorphism and on $\widetilde{V}_{\alpha,i}$, we have: $\mathrm{d}(\widetilde{g}_\alpha)=\mathrm{d}F$. Hence, there exists a constant $C_{\alpha,i}\in\mathbb{C}$ such that $F=\widetilde{g}_\alpha +C_{\alpha,i}$ on $\widetilde{V}_{\alpha,i}$. Now, it is not difficult to verify that 
$$\left(\widetilde{V}_{\alpha,i}, e^{-C_{\alpha,i}}\widetilde{\psi}_{{\alpha}_{|\widetilde{V}_{\alpha,i}}}\right)_{\alpha\in A; i\in I_\alpha}$$ is a weakly \K\ metric on $\widetilde{Y}$. 

Furthermore, since $\widetilde{\psi}_{{\alpha}_{|\widetilde{V}_{\alpha,i}}}$ are by construction constant on the fibers of $\widetilde{\pi}$, they descend to $\widetilde{X}$, where we denote them by $\widetilde{\varphi}_{{\alpha}_{|\widetilde{U}_{\alpha,i}}}$. Moreover, since $\widetilde{\varphi}_{{\alpha}_{|\widetilde{U}_{\alpha,i}}}=p^\star\varphi_{\alpha_{|\widetilde{U}_{\alpha,i}}}$ and $p$ is a local biholomorphism, they are strictly plurisubharmonic, hence the family 
$$\widetilde{\omega}=\left(\widetilde{U}_{\alpha,i}, e^{-C_{\alpha,i}}\widetilde{\varphi}_{{\alpha}_{|\widetilde{U}_{\alpha,i}}}\right)_{\alpha\in A; i\in I_\alpha}$$ is a \K\ metric on $\widetilde{X}$.

From this point forward, the proof is similar to the one for manifolds, with the necessary adaptations. 
Knowing that $\widetilde{g}_\alpha$ is constant on the fibers of $\widetilde{\pi}$, we deduce that $\widetilde{g}_\alpha$ descends to a function $\widetilde{h}_\alpha$ on $p^{-1}(U_\alpha)=\cup_{i\in I_{\alpha}}\widetilde{U}_{\alpha,i}$. Also, since $F=\widetilde{g}_\alpha +C_{\alpha,i}$ on $\widetilde{V}_{\alpha,i}$, it follows that $F$ descends to a function $f$ on $\widetilde{X}$.
Hence, we have $f=\widetilde{h}_\alpha +C_{\alpha,i}$ on $\widetilde{U}_{\alpha,i}$.

Consider $\gamma\in\Aut_X{\widetilde{X}}$.
By the commutativity of the diagram, we also have $\widetilde{h}_\alpha=h_\alpha\circ p$, hence it is invariant to the action of $\gamma$. That being so, taking into account that $\widetilde{h}_\alpha=f-C_{\alpha,i}$ on $\widetilde{U}_{\alpha,i}$, we get that the 1-form $\mathrm{d}f$ defined on $\widetilde{X}_{\reg}$ is invariant to the action of $\gamma$.
Since $\mathrm{d}(f-\gamma^\star f)=\mathrm{d}f-\mathrm{d}(\gamma^\star f)=\mathrm{d}f-\gamma^\star(\mathrm{d} f)=0$ on $\widetilde{X}_{\reg}$, there exists $C\in\mathbb{C}$ such that $f=\gamma^\star f+C$ on $\widetilde{X}_{\reg}$. By the continuity of $f$ and the connectedness and density of $\widetilde{X}_{\reg}$, we deduce that $f=\gamma^\star f+C$ on $\widetilde{X}$.
Next, we want to see how $\gamma$ acts on the \K\ metric. For each $j\in I_\alpha$, there is exactly one $i\in I_\alpha$ such that $\gamma(\widetilde{U}_{\alpha,j})=\widetilde{U}_{\alpha,i}$, thus:
$$\gamma^\star\left(e^{-C_{\alpha,i}}\widetilde{\varphi}_{{\alpha}_{|\widetilde{U}_{\alpha,i}}}\right)=
\gamma^\star\left(e^{\widetilde{h}_\alpha-f}\widetilde{\varphi}_{{\alpha}_{|\widetilde{U}_{\alpha,i}}}\right)=
e^{\widetilde{h}_\alpha-\gamma^\star f}\widetilde{\varphi}_{{\alpha}_{|\widetilde{U}_{\alpha,j}}}=
e^{C}e^{\widetilde{h}_\alpha-f}\widetilde{\varphi}_{{\alpha}_{|\widetilde{U}_{\alpha,j}}}$$
$$=e^C\left(e^{-C_{\alpha,j}}\widetilde{\varphi}_{{\alpha}_{|\widetilde{U}_{\alpha,j}}}\right),$$
which consequently gives $\gamma^\star\widetilde{\omega}=e^C\widetilde{\omega}$, ending the proof for the direct implication.

Now, in order to prove the reversed implication, we suppose that the universal cover $\widetilde{X}$ of the complex space $X$, has a \K\ metric $\widetilde{\omega}$ such that $\Aut_X\widetilde{X}$ acts on $\widetilde{X}$ by \K\ homotheties. Hence, we have the character morphism $\chi:\Aut_X\widetilde{X}=\pi_1(X)\rightarrow \mathbb{R}^{>0}$ given by $\Aut_X\widetilde{X}\ni \gamma\mapsto \frac{\gamma^\star\widetilde{\omega}}{\widetilde{\omega}}\in \mathbb{R}$. On $\widetilde{X}\times \mathbb{R}$ we consider the following equivalence relation: $(x,t)\sim(y,s)$ if there exists $\gamma\in\Aut_X\widetilde{X}$ such that $y=\gamma(x)$ and $s=\chi(\gamma)t$. Then, $E=((\widetilde{X}\times\mathbb{R})/{\sim})\rightarrow X$ is a line bundle which is trivial, since there exists an open cover of $X$ and a choice of transition maps which are all positive. 
Given a section $u$ on $E$ which is non-zero at every point, we obtain a section $\widetilde{u}=p^\star u$ for the line bundle $\widetilde{E}=p^\star E$ on $\widetilde{X}$, which is also trivial. Hence, we may consider that $\widetilde{u}$ is a function with values in $\mathbb{R}^{>0}$.
For any $\gamma \in \Aut_X\widetilde{X}$, by the construction of the line bundle $\widetilde{E}$, we get 
$$\frac{\gamma^\star\widetilde{u}}{\widetilde{u}}=\chi(\gamma)=\frac{\gamma^\star\widetilde{\omega}}{\widetilde{\omega}},$$
hence $\displaystyle\frac{1}{\widetilde{u}}\widetilde{\omega}$ is deck-invariant.
There exists a real function $\widetilde{h}$ such that $\displaystyle\frac{1}{\widetilde{u}}=e^{\widetilde{h}}$. Also, we may consider that $\widetilde{\omega}=(\widetilde{U}_{\alpha,i}, \widetilde{\varphi}_{\alpha,i})$, where for every $\alpha$, the family $(\widetilde{U}_{\alpha,i})_{i\in I_\alpha}$ is made of connected open sets which are projected by $p$ biholomorphically on $U_\alpha\subset X$. 
Then, for every $\alpha$, we choose an arbitrary $i\in I_\alpha$ and denote $\varphi_\alpha=\widetilde{\varphi}_{\alpha,i}\circ p_{|\widetilde{U}_{\alpha,i}}^{-1}:U_\alpha\rightarrow \mathbb{R}$, and $h_\alpha=\widetilde{h}\circ p_{|\widetilde{U}_{\alpha,i}}^{-1}:U_\alpha\rightarrow \mathbb{R}$. With these notations, it is now easy to verify that $(U_\alpha, \varphi_\alpha, h_\alpha)_{\alpha\in A}$ is an LCK metric on $X$.
\end{proof}

\subsection{The blow-up at a point of an LCK complex space}

A classical result, by Tricerri \cite{TRI} and Vuletescu \cite{VUL}, says that the blow-up at a point of an LCK manifold is also an LCK manifold. In the next lines we show that this result can be easily generalized to complex spaces.

\begin{proof}[Proof of \ref{TH2}]
Let $X$ be a complex space with the LCK metric $\omega=(V_\alpha, \varphi_\alpha, h_\alpha)_{\alpha\in A}$, and $x_0\in X$. We may assume that $x_0$ has a neighborhood $U$ such that $U\subset V_\alpha$, and $U\cap V_\beta=\emptyset$ for all $\beta\not =\alpha$. We may also assume that $V_\alpha$ is sufficiently small such that it can be embedded as a closed complex subspace in the unit ball $\mathbb{B}\subset\mathbb{C}^N$, and such that $\varphi_\alpha$ extends to a strictly plurisubharmonic function on $\mathbb{B}$ (we keep the same notation $\varphi_\alpha$ for the extended function). 

Now, denote by $\widehat{\mathbb{B}}$ the blow-up of $\mathbb{B}$ in $x_0$, and by $\pi:\widehat{\mathbb{B}}\rightarrow \mathbb{B}$ the projection. Then, $\pi^{-1}(x_0)=:E\simeq \mathbb{P}^{N-1}(\mathbb{C})$. 
Using the technique from \cite{VUL}, one can construct a (1,1)-form $\Omega_E$ as the curvature of a line bundle on $\widehat{\mathbb{B}}$, such that $\Omega_E$ has the following properties:
it is negative definite along $E$ (i.e. $\Omega_E(v,J_{\widehat{\mathbb{B}}}v)<0$ for every $P\in E$ and every non-zero vector $v\in T_P(E)$, where $J_{\widehat{\mathbb{B}}}$ is the complex structure on $\widehat{\mathbb{B}}$), it is  negative semi-definite at points of $E$ (i.e. $\Omega_E(v,J_{\widehat{\mathbb{B}}}v)\leq 0$ for every $P\in E$ and every $v\in T_P(\widehat{\mathbb{B}})$), and it is zero outside a compact subset of $\pi^{-1}(U)$.

Then, for a sufficiently small constant $\varepsilon>0$, the (1,1)-form 
$$h=i\partial\overline{\partial}(\pi^\star\varphi_\alpha) -\varepsilon\Omega_E$$ 
is positive definite. It is also $\mathrm{d}$-closed, since $\Omega_E$ is the curvature form of a line bundle. Hence, it is a \K\ form on $\widehat{\mathbb{B}}$.  By Dolbeault's lemma, it can be represented in the form of the generalized definition of \K\ metrics, as $h=(W_{j}, \varphi_{j})_{j\in J}$. With the notations $V_{\alpha j}=\pi^{-1}(V_\alpha)\cap W_j$ and $\varphi_{\alpha j}={\varphi_j}_{|V_{\alpha j}}$, we have that $(V_{\alpha j},\varphi_{\alpha j})_{j\in J}$ is a \K\ metric on $\widehat{V_\alpha}$, the blow-up of $V_\alpha$ at $x_0$.
Since $\Omega_E=0$ outside $\pi^{-1}(U)$, the strictly plurisubharmonic functions $\pi^\star \varphi_\alpha$ and $\varphi_{\alpha j}$ determine the same \K\ metric on $V_{\alpha j}\setminus \pi^{-1}(U)$.
For this reason, by glueing $X\setminus U$ and $\widehat{V_\alpha}$ in the natural way,
with the notation $h_{\alpha j}={(\pi^\star h_\alpha)}_{|V_{\alpha j}}$, the compatibility condition $e^{h_\beta}i\partial\overline{\partial}\varphi_\beta=e^{h_{\alpha j}}i\partial\overline{\partial}\varphi_{\alpha j}$ holds on $V_\beta\cap V_{\alpha j}\cap X_{\reg}$, for any $j\in J$ and any $\beta\in A\setminus\{\alpha\}$, because $V_\beta\cap\pi^{-1}(U)=\emptyset$.
Hence,
$$(V_\beta, \varphi_\beta, h_\beta)_{\beta\in A\setminus\{\alpha\}}\cup (V_{\alpha j},\varphi_{\alpha j},h_{\alpha j})_{j\in J}$$ 
is an LCK metric on $\widehat{X}$, the blow-up at $x_0$ of $X$.

\end{proof}

\section{Examples}

In this section we give examples of LCK complex spaces which do not admit \K\ metrics. They are obtained as the  quotient of an LCK (non-\K) manifold by a finite group of automorphisms which have fixed points.

\begin{example}\label{EX1}
Quotients of Hopf manifolds of dimension at least 3.

Consider $\lambda\in \mathbb{C}$, $\lambda\not=0$, $|\lambda|<1$ and the matrix $A=\diag(\lambda,\lambda,\lambda^2)$. Denote $G=\{A^n : n\in\mathbb{Z}\}$. Then, the Hopf manifold $H=(\mathbb{C}^3\setminus\{0\})/G$ is an LCK manifold which does not admit \K\ metrics.  
Also, take the matrix $B=\diag(-1,-1,1)$ and denote $J=\{I_3, B\}$. 
Define the function $\Phi:\mathbb{C}^3\rightarrow\mathbb{C}^4$ by $\Phi(z_1,z_2,z_3)=(z_1^2,z_2^2,z_1z_2,z_3)$.
By the results in \cite[Section 4]{CAR}, $Y=(\mathbb{C}^3\setminus\{0\})/J$ is a singular complex space biholomorphic to $Y_0=\Phi(\mathbb{C}^3\setminus\{0\})\subset\mathbb{C}^4\setminus\{0\}$. 
Consider the function $\varphi(w)=\|w\|^2$ on $\mathbb{C}^4$. Then, $i\partial\overline{\partial}\varphi$ is a \K\ form on $\mathbb{C}^4$ which induces a \K\ metric $\Omega_0$ on $Y_0$. Denote by $\Omega$ the \K\ metric on $Y$ obtained via the biholomorphism $Y \simeq Y_0$. Also, denote $Y/G=:X$.
It is not difficult to verify that  $\Aut_X Y\simeq G$ acts by homotheties of the \K\ metric $\Omega$. 
Finally, by theorem \ref{TH1}, which for the converse implication is true (with the same proof) for any cover, not only the universal cover,
$X$ is LCK. However, $X$ does not admit \K\ metrics, since it contains $(\mathbb{C}^2\setminus\{0\}\times\{0\})/JG$ as a closed complex subspace which is a 2-dimensional Hopf mainfold.
\end{example} 

\begin{example}
Quotients of compact LCK surfaces.

Let $(M,\Omega)$ be a compact LCK (non-\K ) manifold which has a finite cyclic group $G=\langle F\rangle\subset\Aut(M)$ for which the metric $\Omega$ is invariant, and such that the fixed point locus of $F$ is a finite set. For every point $x\in M$, there exists a neighborhood $U\ni x$ in $M$ such that on $U$, $\Omega=e^{-f} i \partial\overline{\partial}\phi$, where $f$ is smooth and $\phi$ is strictly plurisubharmonic. 
If the metric on $U$ can also be written $\Omega=e^{-g}i\partial\overline{\partial}\psi$, then $e^{-(f-g)}i\partial\overline{\partial}\phi=\partial\overline{\partial}\psi$ leads to $d(f-g)=0$, hence $f=g+C$, which further implies $e^{-f} i \partial\overline{\partial}\phi=e^{-f} i \partial\overline{\partial}(e^C\psi)$. Therefore, we may assume that $f$ is $G$-invariant.
Moreover, by taking the pull-back of the metric $\Omega$ by all the elements of $G$ and then taking the average metric, we may assume that $\varphi$ is also $G$-invariant, hence both $f$ and $\varphi$ descend to functions on the singular space $X=M/G$. 

Consequently, the LCK metric $\Omega$ descends to $\omega=(V_j, \varphi_j, h_j)_{1\leq j\leq r}$, which is an LCK metric on $X_{\reg}$. However, the functions $\varphi_j$ and $h_j$ may not be smooth at the singular points of $X$.
We may assume that the projection on $X$ of every fixed point of $F$ has a small neighborhood in $X$ which intersects only one of the sets $V_j$. We can modify both $\varphi_j$ and $h_j$ on this small neighborhood, to make them smooth, with the modified $\varphi_j$ still strictly plurisubharmonic, thus obtaining a modified metric on $X$ which is still LCK. 

Now, if we assume that $X$ has a \K\ metric, then its pull-back to $M$ is a (1,1)-form on $M$ which is \K\ on the complement of the set of fixed points of $F$, and it can be modified at those points to obtain a \K\ metric on the whole $M$, yielding a contradiction. Hence, the singular complex space $X$ does not admit \K\ metrics.
\end{example}

\section{Remarks}

Mumford \cite{MUM} proved that if $P\in V$ is a normal point of a 2-dimensional algebraic variety, then $V\setminus\{P\}$ is locally simply connected around $P$ if and only if $P$ is a regular point of $V$. Thus, the regular locus of a simply connected normal complex space is not, in general, simply connected.
For this reason, the proof of our theorem cannot be modified to use normalizations instead of desingularizations, which would have been better, since the method would have worked for any complex space. 
But it is worth remarking that even if for our proof the additional assumption on the type of singularities is essential, the theorem might be true for the general case of singular complex spaces. Thus, we propose the following problem:

\begin{problem}
Prove that a complex space $X$ admits an LCK metric if and only if its universal cover $\widetilde{X}$ admits a \K\ metric such that deck automorphisms act on $\widetilde{X}$ by homotheties of the \K\ metric, or find a counterexample to this statement.
\end{problem}

\vspace{0.4cm}
We also want to point out that the Hopf surface (2-dimensional) and the Inoue surface do not have ``enough" automorphisms to be used for examples like \ref{EX1}, and the quotient of any of these surfaces by a finite group of automorphisms is again smooth and in the same class as the initial surface. Also, since all our concrete examples are quotients of LCK, non-K\" ahler manifolds, it would be interesting to find a different way to construct an example of LCK complex space which does not admit \K\ metrics.

\subsection*{Acknowledgment}
Ovidiu Preda is grateful to Professor Liviu Ornea for guiding his learning of geometry which led to the problem studied in this article. Both authors are thankful to Alexandra Otiman for repeatedly taking the time to answer their questions.


\end{document}